\documentclass[12pt,reqno,oneside]{amsproc}
\title[Weak Solutions for the 3D~Navier-Stokes-plate system]{A global existence result on weak solutions for the 3D~Navier-Stokes-plate system with no contact}

\author[M.~Bukal]{Mario Bukal}
\address{University of Zagreb, Faculty of Electrical Engineering and Computing, Unska 3, 10000 Zagreb, Croatia}
\email{mario.bukal@fer.hr}

\author[I.~Kukavica]{Igor Kukavica}
\address{Department of Mathematics, University of Southern California, Los Angeles, CA 90089}
\email{kukavica@usc.edu}

\author[L.~Li]{Linfeng Li}
\address{Department of Mathematics, University of California Los Angeles, Los Angeles, CA 90095}
\email{lli265@math.ucla.edu}

\author[B.~Muha]{Boris Muha}
\address{University of Zagreb, Faculty of Science, Department of Mathematics, Bijeni\v cka cesta 30, 10000 Zagreb, Croatia}
\email{borism@math.hr}

\usepackage[normalem]{ulem}
\usepackage{enumitem}
\usepackage{bbm}
\usepackage{fancyhdr}
\usepackage{comment}
\usepackage{tikz}
\chardef\forshowkeys=0
\chardef\refcheck=0
\chardef\showllabel=0
\chardef\sketches=0
\chardef\showcolors=1

\usepackage{color}
\usepackage{bbm,yfonts}
\usepackage{upgreek}
\usepackage{enumitem}
\usepackage{datetime}
\usepackage{esint}
\usepackage{fancyhdr}
\usepackage{comment}
\usepackage[utf8]{inputenc}
\allowdisplaybreaks

\let\pa\partial   
  
\let\eps\varepsilon  

\newcommand{\N}{{\mathbb N}}    
\newcommand{\R}{{\mathbb R}}  
\newcommand{\Z}{{\mathbb Z}}
\newcommand{\diver}{\operatorname{div}}
\newcommand{\Rey}{\operatorname{Re}}
\newcommand{\dist}{\operatorname{dist}}

\newcommand{\D}{\operatorname{D}}

\newcommand{\dd}{\,d}

\newcommand{\red}{\textcolor{red}}

\newcommand{\nablax}{\nabla_x}
\newcommand{\Deltax}{\Delta_x}

\ifnum\forshowkeys=1

\usepackage[notcite,color]{showkeys}
\fi

\usepackage[margin=1in]{geometry}
\usepackage{amsmath, amsthm, amssymb}
\usepackage{times}
\usepackage{graphicx}
\usepackage{marginnote}
\usepackage[unicode,breaklinks=true,colorlinks=true,linkcolor=black,urlcolor=black,citecolor=black]{hyperref}

\ifnum\refcheck=1
\usepackage{refcheck}
\fi

\begin{document}
	\def\ch{{\hat{c}}}
	\def\bp{\bar{\partial}}
	\def\bbeta{\beta}
	\def\tt{{\mathbb{T}^2}}
	\def\ttt{t_\ast}
	\def\xxx{x_\ast}
	\def\iii{I_\ast}
	\def\bbb{B_\ast}
	\def\aa{\nu}
	\def\etah{\eta}
	\def\uh{u}
	\def\vh{v}
	\def\qh{q}
	\def\ph{p}
	\def\inp{{\colr IN PROGRESS}}
	\def\rth{{\colr REVISED UP TO HERE}}
	\def\fNS{f}
	\def\fP{g}
	\def\bnew{\colr {\bf }}
	\def\enew{\colb {}}
	\def\bold{\colu {\bf }}
	\def\eold{\colb {}}
	\def\inte{\int_{{\mathbb{T}^2}_\epsilon}}
	\def\intb{\int_{\Gamma_\epsilon}}
	\def\D{D}
	\def\CommT{{R}}
	\def\BndT{{B}}

	\def\XX{X}
	\def\YY{Y}
	\def\ZZZ{Z}
	
	\def\intint{\int    \int}
	\def\OO{\mathcal O}
	\def\SS{\mathbb S}
	\def\CC{\mathbb C}
	\def\N{\mathbb N}
	\def\RR{\mathbb R}
	\def\TT{\mathbb T}
	\def\ZZ{\mathbb Z}
	\def\HH{\mathbb H}
	\def\RSZ{\mathcal R}
	\def\LL{\mathcal L}
	\def\SL{\LL^1}
	\def\ZL{\LL^\infty}
	\def\GG{\mathcal G}
	\def\erf{\mathrm{Erf}}
	\def\mgt#1{\textcolor{magenta}{#1}}
	\def\ff{\rho}
	\def\gg{G}
	\def\sqrtnu{\sqrt{\nu}}
	\def\ww{w}
	\def\ft#1{#1_\xi}
	\def\lec{\lesssim}
	\def\les{\lesssim}
	\def\gec{\gtrsim}
	\renewcommand*{\Re}{\ensuremath{\mathrm{{\mathbb R}e  }}}
	\renewcommand*{\Im}{\ensuremath{\mathrm{{\mathbb I}m  }}}
	
	\ifnum\showllabel=1
	\def\llabel#1{\marginnote{\color{lightgray}\rm\small(#1)}[-0.0cm]\notag}
	\else
	\def\llabel#1{\notag}
	\fi
	
	\newcommand{\norm}[1]{\left\|#1\right\|}
	\newcommand{\nnorm}[1]{\lVert #1\rVert}
	\newcommand{\abs}[1]{\left|#1\right|}
	\newcommand{\NORM}[1]{| | | #1| | |}
	
	\newtheorem{theorem}{Theorem}[section]
	\newtheorem{Theorem}{Theorem}[section]
	\newtheorem{corollary}[theorem]{Corollary}
	\newtheorem{Corollary}[theorem]{Corollary}
	\newtheorem{proposition}[theorem]{Proposition}
	\newtheorem{Proposition}[theorem]{Proposition}
	\newtheorem{Lemma}[theorem]{Lemma}
	\newtheorem{lemma}[theorem]{Lemma}
	
	\theoremstyle{definition}
	\newtheorem{definition}{Definition}[section]
	\newtheorem{remark}[theorem]{Remark}

	\def\theequation{\thesection.\arabic{equation}}
	\numberwithin{equation}{section}

	\definecolor{mygray}{rgb}{.6,.6,.6}
	\definecolor{myblue}{rgb}{9, 0, 1}
	\definecolor{colorforkeys}{rgb}{1.0,0.0,0.0}

	\newlength\mytemplen
	\newsavebox\mytempbox
	\def\weaks{\text{            weakly-* in }}
	\def\weak{\text{            weakly in }}
	\def\inn{\text{            in }}
	\def\cof{\mathop{\rm cof  }\nolimits}
	\def\Dn{\frac{\partial}{\partial N}}
	\def\Dnn#1{\frac{\partial #1}{\partial N}}
	\def\tdb{\tilde{b}}
	\def\tda{b}
	\def\qqq{u}
	\def\lat{\Delta_{\bx}}
	\def\biglinem{\vskip0.5truecm\par==========================\par\vskip0.5truecm}
	
	\def\inon#1{\hbox{\ \ \ \ \ \ \ }\hbox{#1}}               
	\def\onon#1{\inon{on~$#1$}}
	\def\inin#1{\inon{in~$#1$}}
	
	\def\FF{F}
	\def\andand{\text{\indeq and\indeq}}
	\def\ww{w(y)}
	\def\ee{\epsilon_0}
	\def\startnewsection#1#2{ \section{#1}\label{#2}\setcounter{equation}{0}}   
	\def\nnewpage{ }
	\def\sgn{\mathop{\rm sgn  }\nolimits}    
	\def\Tr{\mathop{\rm Tr}\nolimits}    
	\def\div{\mathop{\rm div}\nolimits}
	\def\curl{\mathop{\rm curl}\nolimits}
	\def\dist{\mathop{\rm dist}\nolimits}  
	\def\supp{\mathop{\rm supp}\nolimits}
	\def\indeq{\quad{}}           
	\def\period{.}                       
	\def\semicolon{  ;}
	\def\pa{\partial}            
	\def\pt{\partial_t}                
	
	\ifnum\showcolors=1
	\def\colr{\color{red}}
	\def\colc{\color{cyan}}
	\def\colrr{\color{black}}
	\def\colb{\color{black}}
	\definecolor{colorgggg}{rgb}{0.1,0.5,0.3}
	\definecolor{colorllll}{rgb}{0.0,0.7,0.0}
	\definecolor{colorhhhh}{rgb}{0.3,0.75,0.4}
	\definecolor{colorpppp}{rgb}{0.7,0.0,0.2}
	\definecolor{coloroooo}{rgb}{0.45,0.0,0.0}
	\definecolor{colorqqqq}{rgb}{0.1,0.7,0}
	\def\coly{\color{lightgray}}
	\def\colg{\color{colorgggg}}
	\def\collg{\color{colorllll}}
	\def\cole{\color{coloroooo}}
	\def\coll{\color{colorqqqq}}
	\def\coleo{\color{colorpppp}}
	\def\colu{\color{blue}}
	\def\colc{\color{colorhhhh}}
	\def\colW{\colb}   
	\definecolor{coloraaaa}{rgb}{0.6,0.6,0.6}
	\def\colw{\color{coloraaaa}}
	\else
	\def\colr{\color{black}}
	\def\colrr{\color{black}}
	\def\colb{\color{black}}
	\def\coly{\color{black}}
	\def\colg{\color{black}}
	\def\collg{\color{black}}
	\def\cole{\color{black}}
	\def\coleo{\color{black}}
	\def\colu{\color{black}}
	\def\colc{\color{black}}
	\def\colW{\color{black}}
	\def\colw{\color{black}}
	\fi

	\def\comma{ {\rm ,\qquad{}} }            
	\def\commaone{ {\rm ,\quad{}} }          
	\def\nts#1{{\color{red}\hbox{\bf ~#1~}}} 
	\def\ntsf#1{\footnote{\color{colorgggg}\hbox{#1}}} 
	\def\blackdot{{\color{red}{\hskip-.0truecm\rule[-1mm]{4mm}{4mm}\hskip.2truecm}}\hskip-.3truecm}
	\def\bluedot{{\color{blue}{\hskip-.0truecm\rule[-1mm]{4mm}{4mm}\hskip.2truecm}}\hskip-.3truecm}
	\def\purpledot{{\color{colorpppp}{\hskip-.0truecm\rule[-1mm]{4mm}{4mm}\hskip.2truecm}}\hskip-.3truecm}
	\def\greendot{{\color{colorgggg}{\hskip-.0truecm\rule[-1mm]{4mm}{4mm}\hskip.2truecm}}\hskip-.3truecm}
	\def\cyandot{{\color{cyan}{\hskip-.0truecm\rule[-1mm]{4mm}{4mm}\hskip.2truecm}}\hskip-.3truecm}
	\def\reddot{{\color{red}{\hskip-.0truecm\rule[-1mm]{4mm}{4mm}\hskip.2truecm}}\hskip-.3truecm}
	
	\def\tdot{{\color{green}{\hskip-.0truecm\rule[-.5mm]{3mm}{3mm}\hskip.2truecm}}\hskip-.1truecm}
	\def\gdot{\greendot}
	\def\bdot{\bluedot}
	\def\ydot{\cyandot}
	\def\rdot{\cyandot}
	\def\fractext#1#2{{#1}/{#2}}
	\def\ii{\hat\imath}
	\def\boris#1{\textcolor{blue}{#1}}
	\def\vlad#1{\textcolor{cyan}{#1}}
	\def\igor#1{\text{{\textcolor{colorqqqq}{#1}}}}
	\def\linfeng#1{\textbf{\textcolor{colorqqqq}{LL:~#1}}}
	\def\igorf#1{\footnote{\text{{\textcolor{colorqqqq}{#1}}}}}
	\def\mario#1{\textcolor{purple}{#1}}

	\newcommand{\myr}[1]{{\color{red} #1 }}%
	\newcommand{\uvro}{\upvarrho}
	\newcommand{\bv}{\boldsymbol v}
	\newcommand{\bV}{\boldsymbol V}
	\newcommand{\bu}{\boldsymbol u}
	\newcommand{\bs}{\boldsymbol}
	\newcommand{\bx}{x}
	\newcommand{\obx}{\overline{x}}
	\newcommand{\bw}{w}
	\newcommand{\bn}{\mathbf n}
	\newcommand{\by}{\boldsymbol y}
	\newcommand{\bz}{\boldsymbol z}
	\newcommand{\bef}{\boldsymbol f}
	\newcommand{\material}{{\mathbb{T}^2}_{\eps,h}}
	\newcommand{\vol}{\operatorname{vol}}
	\newcommand{\bog}{{\rm Ext}_{\diver}}
	\newcommand{\p}{\partial}
	\newcommand{\UE}{U^{\rm E}}
	\newcommand{\PE}{P^{\rm E}}
	\newcommand{\KP}{K_{\rm P}}
	\newcommand{\uNS}{u^{\rm NS}}
	\newcommand{\vNS}{v^{\rm NS}}
	\newcommand{\pNS}{p^{\rm NS}}
	\newcommand{\uE}{u^{\rm E}}
	\newcommand{\vE}{v^{\rm E}}
	\newcommand{\pE}{p^{\rm E}}
	\newcommand{\ua}{u_{\rm   a}}
	\newcommand{\ue}{u_{\rm   e}}
	\newcommand{\ve}{v_{\rm   e}}
	\newcommand{\omegaL}{{\mathbb{T}^2_L}}  
	
	\def\red#1{\textcolor{red}{#1}}

\begin{abstract}
We consider the three-dimensional fluid-structure interaction system
modeling a system consisting of a viscous incompressible fluid and an elastic plate forming its moving upper boundary. 
The fluid is described by the incompressible Navier-Stokes equations with a free upper boundary that evolves according to the motion of the structure, coupled via the velocity- and stress-matching conditions. 
We show that under a rather general condition on the initial data,
there exists a global-in-time weak solution of the system. In
particular, there is no contact between the plate and the bottom boundary.
\end{abstract}

\maketitle

\setcounter{tocdepth}{2} 

\section{Introduction}\label{sec01} 

\subsection{Motivation and outline}
Fluid-structure interaction (FSI) phenomena arise in a wide range of applications: from biological systems \cite{Tit94,HHS08,BGN14} and geophysical flows \cite{LPN13,TsRi12} to industrial processes~\cite{Dow15,YGJ17}. In particular, we emphasize the importance of microfluidic systems with compliant boundaries, which play a central role in the behavior of modern soft-walled micro-channels and are essential components of emerging lab-on-chip technologies~\cite{SSA04,DaFin06}. Experiments and theoretical analyses have revealed substantial physical insights in microfluidic systems, including peeling, healing and bursting \cite{HoMa04}, suppression of viscous fingering \cite{PIHJ12}, pressure-flow relationships \cite{GAGJ}, explicit pressure-deformation laws \cite{CCSS}, dynamical instabilities with modification of critical Reynolds numbers \cite{VK}, etc. 

Despite all this, the mathematical theory of such FSI systems remains incomplete. Coupling the incompressible Navier-Stokes equations with a moving elastic plate boundary yields a strongly nonlinear free-boundary problem in which the fluid domain evolves according to the structural displacement. This raises fundamental questions concerning the long-time behavior of solutions, the exchange of energy between fluid and structure, and the possibility of contact between the elastic wall and the opposing rigid boundary. 

Motivated by these questions, in the present work we study a canonical FSI model consisting of the incompressible Navier-Stokes equations in a time-dependent domain coupled to a linearly elastic plate forming upper part of the fluid  boundary. Based on the energy dissipation inequality, we identify sufficient conditions on initial data that rule out possibility of a contact in finite time, and therefore establish the global-in-time existence of weak solutions. This is the first occurrence of the global-in-time existence of weak solutions to the FSI problem in the literature. As such, it provides a rigorous mathematical foundation for the operation of microfluidic systems that aligns with experimentally observed robustness in compliant channels~\cite{CJFY}.

Below we provide a brief literature review. The FSI model under consideration is formulated in Section~\ref{sec:FSI} in detail, and the main result (Theorem~\ref{T01}) is stated there. In Section~\ref{sec:proof} we present the proof of the main result, while in Section \ref{sec:app} we discuss applications of our result to microfluidic systems.

\subsection{Brief literature review}
A major difficulty in establishing global existence of weak solutions is the possibility of contact between the deforming structure and the rigid bottom boundary in finite time.
Existence of weak solutions up to such a possible contact has been obtained in various contexts; see
\cite{CDEG,G,MC1,MC2,MS,TW}.
In two dimensions, Grandmont and Hillairet in \cite{GH} proved global strong solution for a viscoelastic beam-fluid interaction system by first ruling out contact and then propagating regularity. 
More recently, Breit and Roy in \cite{BR} proved a non-contact result for weak solutions of a compressible beam-fluid model under additional regularity assumptions, and
the authors of this paper extended the non-contact result to the three-dimensional incompressible plate-fluid system in \cite{BKLM}, thereby providing a positive answer to the question raised in~\cite[Remark~4.6]{BR}.
In contrast to our previous work \cite{BKLM}, where global existence was established conditionally on the additional regularity of weak solutions, the present paper proves global existence of a weak solution for the unconditional Cauchy problem under a condition that imposes a constraint on the initial energy compared to the average of the initial position of the plate, while allowing arbitrarily large initial energy and arbitrary proximity of the structure to the bottom.
Finally, we note that small-data global existence has been established in several related FSI models; see \cite{Co, IKLT1,IKLT2,KO1,KO2,MT,QGY}.

\section{Problem formulation and main result}
\label{sec:FSI}
\subsection{FSI problem}
At each time $t\geq0$, the fluid occupies the domain $\Omega_{\eta}(t)$, which is a subgraph of a space-time dependent function $\etah\colon {\mathbb{T}^2}\times[0,+\infty)\to {\mathbb R}$, representing the height of the moving boundary.
Here ${\mathbb{T}^2}=\R^2/\Z^2$ denotes the two-dimensional unit torus.
Thus, for any $t\geq0$, the domain is given by
\begin{equation*}
{\Omega}_{\eta} (t)=
\{ (\bx, z) : \bx=(x_1,x_2) \in {\mathbb{T}^2},
z\in (0,\eta(\bx,t))\}\subseteq\R^3.
\end{equation*}
The fluid is incompressible and Newtonian, and its motion is governed by the incompressible Navier-Stokes equations
\begin{align}
\begin{split}
\partial_{t} u - \Delta u + u\cdot \nabla u + \nabla p &= 0 \onon{\Omega_{\eta} (t) \times (0,T)}, \\
\diver u &= 0 \onon{\Omega_{\eta} (t) \times (0,T)},
\end{split}
\label{EQNS}
\end{align}
where $u\colon \Omega_{\eta} (t) \times (0,T) \to \mathbb{R}^3$ and $p\colon \Omega_{\eta} (t) \times (0,T) \to \mathbb{R}$ denote the velocity and pressure.
As it is customary in fluid-structure interaction analysis, we abbreviate $\Omega_{\eta} (t) \times (0,T)$ to mean $\bigcup_{t \in (0,T)} \Omega_{\eta} (t) \times \{t\}$.

The displacement $\eta(\bx,t)$ of the elastic boundary is governed by the fourth-order damped plate equation
\begin{equation}
\etah_{tt}
-
\aa
\Delta_{\bx} \etah_t
+
\Delta_{\bx}^2 \etah
= g \onon{\tt\times (0,T)}
,
\label{EQPL}
\end{equation}
where $\Delta_{\bx}=\partial_{11}+\partial_{22}$ is the horizontal Laplacian and $\aa\geq0$ is a constant.
The case $\aa>0$ corresponds to a visco-elastic plate \cite{CR}, while $\aa=0$ recovers the classical elastic plate.
The forcing term $g(\bx,t)$ represents the vertical fluid traction acting on the structure.
Specifically,
\begin{equation}
g
=
-S_{\eta}
(-p I+\nabla u) (\bx, \eta (\bx,t),t)
n^{\eta}\cdot e_3
\onon{\tt\times (0,T)},
\label{EQ82a}
\end{equation}
where $S_{\eta} (\bx,t)=\sqrt{1+|\nablax \eta (\bx,t)|^2}$ is the surface element of the deformed plate, and $n^{\eta}$ is the unit outward normal to the graph of $\eta(\bx,t)$, given by
\begin{align}
n^{\eta}
=
\frac{1}{\sqrt{1+ |\nablax \eta|^2}}
(-\partial_{1} \eta, -\partial_{2} \eta, 1).
\llabel{EQ83}
\end{align}

We impose the no-slip condition at the bottom of the fluid layer and require the velocity to match the motion of the elastic boundary on the top:
\begin{align}
\begin{split}
u(\bx,0,t)&=0
\onon{\tt \times (0,T)},
\\
u(\bx, \eta(\bx,t),t)
&=(0,0,\etah_t(\bx,t))
\onon{\tt\times (0,T)}.
\label{EQ26}
\end{split}
\end{align}
Periodic boundary conditions are assumed along the lateral boundaries.

The system is complemented with initial conditions
\begin{align}
u(x,z,0)= u_0(x,z), \quad
\eta (x,z,0) = \eta_0 (x,z),
\quad
\pt \eta (x,z,0) = \eta_1 (x,z).
\label{EQ27}
\end{align}

The objective of this paper is to establish global existence of weak solutions under an assumption on the initial data.

\subsection{The main result}
First, we define the following function spaces:
\begin{align}
\begin{split}
	&
	V_{\eta}=\{u \in H^1(\Omega_{\eta} (t)): \diver u =0, {u(x,0,t)=0 \text{~on~} \TT^2 \times (0,T)}\},
	\\&
	V_F= 
	L^\infty ((0,T), L^2 (\Omega_{\eta } (t)))
	\cap L^2 ((0,T), V_{\eta} (t)),
	\\& 
	V_K= L^\infty ((0,T), H^2 (\TT^2))
	\cap W^{1,\infty} ((0,T), L^2 (\TT^2)),
		\\& 
	V_S= \{(u,\eta)\in V_F\times V_K: u(x,\eta (x,t),t) = \partial_t \eta (x,t) e_3, ~ \forall (x,t)\in \TT^2 \times (0,T)\},
	\\&
	V_T=\{(\Phi,\phi) \in V_F \times V_K: \partial_t\Phi\in L^2((0,T), L^2(\Omega_{\eta}(t))),\; \Phi (x,\eta (x,t),t) = \phi (x,t) e_3, \\ &~\qquad~\qquad~\qquad~\qquad \forall (x,t)\in \TT^2 \times (0,T)\}.
	\llabel{EQ28}
\end{split}
\end{align}

By formally multiplying equations \eqref{EQNS}$_1$ by $u$ and \eqref{EQPL} by $\pa_t\etah$,
integrating over $\Omega_{\eta} (t)$ and $\TT^2$ respectively,
integrating by parts in space, and using the boundary condition \eqref{EQ26}, we obtain
\begin{align}
	\begin{split}
		&
		\frac{1}{2}
		\left(\|u(t)\|^2_{L^2(\Omega_{\eta} (t))}
		+
		\|\Delta_{\bx}\etah(t)\|^2_{L^2(\TT^2)}
		+
		\|\partial_t\etah(t)\|^2_{L^2(\TT^2)} \right)
		+
		\int_0^t\!\!\int_{\Omega_{\etah}(s)}|\nabla u|^2
		\dd \bx \dd z \dd s
		\\&\indeq
		+
		\aa
		\int_0^t
		\int_{\TT^2}
		|\partial_t\nabla_{\bx}
		\eta(s)|^2
		\dd \bx \dd z \dd s
		\leq
		\frac{1}{2}
		\left(	\Vert u_0\Vert_{L^2 (\Omega_\eta (0))}^2
				+ 	\|\Delta_{\bx} \etah_0\|^2_{L^2(\TT^2)}
		+
		\|\etah_1\|^2_{L^2(\TT^2)}
	\right),
		\label{EI}
	\end{split}
\end{align}
where we used the Reynolds transport theorem.

Next, we state the definition of the weak solution of the system \eqref{EQNS}--\eqref{EQ27}.
\begin{definition}
	\label{D01}
{\rm
	We say that $(u,\eta) \in V_S$ is a weak solution of
\eqref{EQNS}--\eqref{EQ27} if it satisfies the energy inequality \eqref{EI}, and for every test functions $(\Phi, \phi)\in V_T$ the equality 
\begin{align}
\begin{split}
	\frac{d}{dt}
	\left(	\int_{\tt}
	\pt \eta \phi 
	+
	\int_{{\Omega}_{\eta} (t)}
	u\cdot \Phi
	\right)
	&
	= \int_{{\Omega}_{\eta} (t)}
	\left(	u \pt \Phi
	+
	u \otimes u : \nabla \Phi
	-\nabla u : \nabla \Phi
	 \right)
	\\&
	\quad
	+ \int_{\tt}
	\left(\pt \eta \pt \phi
	-
	\Deltax \eta \Deltax \phi
	+
	\aa \pt \eta  \Deltax \phi
	\right)
	\label{EQweak}
\end{split}
\end{align}
holds in~$\mathcal{D}'(0,T)$.
}
\end{definition}

The following theorem is our main result.

\begin{theorem}
	\label{T01}
	For every $\kappa \in (0,1)$, there exists a constant $C_\kappa>0$ such that  if
	\begin{align}
		\| u_0 \|^2_{L^2({\Omega}_{\eta}(0))}
		+
		\|\eta_1\|^2_{L^2({\mathbb{T}^2})}
		+
		\|\Delta_{\bx}\eta_0 \|^2_{L^2({\mathbb{T}^2})}
		\leq \frac{\bar{\eta}_0^2}{C_\kappa},
		\label{cond_nonc}
	\end{align}
where $\bar{\eta}_0
=
\int_{\mathbb{T}^2} \eta_0 (x) \dd \bx >0$,
then there exists a global-in-time weak solution $(u,\eta)$ of \eqref{EQNS}--\eqref{EQ27}, and $\eta \geq \kappa\bar{\eta}_0$ on $\tt\times(0,+\infty)$.
\end{theorem}

\begin{remark}
{\rm
\label{R02}
Our proof shows, moreover, that every
weak solution
on $(0,\bar T)$, where $\bar T>0$,
that satisfies
\eqref{cond_nonc}
for some $\kappa\in(0,1)$ and $C_{\kappa}>0$,
cannot have a contact at any $T\in(0,\bar T)$
with $\inf_{t\in[0,\bar T)}\eta>0$.
}
\end{remark}

\begin{remark}
{\rm
	\label{R01}
From the fluid-mechanical point of view, we adopt the Cauchy stress tensor $\sigma (u,p) = -p \mathbb{I}_3 +\nabla u$
	in \eqref{EQNS}--\eqref{EQ27}, i.e., we omit the transposed gradient term~$\nabla^T u$. 
Theorem~\ref{T01} remains valid if the full Cauchy stress, including $\nabla^T u$, is used; see \cite[Remark~1]{GH} for further details.
}
\end{remark}


\section{Proof of Theorem~\ref{T01}}\label{sec:proof}
First, we state the following blow-up alternative.

\begin{proposition}
	\label{P01}
There exists a weak solution $(u,\eta)$ on the time interval $(0,T)$ in the sense of Definition~\ref{D01}.
Furthermore, one of the following statements holds:
	\begin{enumerate}
		\item Either $T=+\infty$, or
		\item $T<\infty$ and a contact occurs at time~$T$.
	\end{enumerate}
\end{proposition}

\begin{proof}[Proof of Proposition~\ref{P01}]
	The construction of a weak solution in this context is by now standard, so we only sketch the argument and refer to the literature for details. We only consider the case $\aa=0$, which is mathematically the most challenging. The case $\aa>0$ produces an additional dissipation term
	$\aa\int_{\mathbb{T}^2} |\nabla_x\partial_t \eta|^2$
	that yields stronger a priori estimates and is therefore advantageous for the construction.

Our case differs from the existing literature only in the use of periodic boundary conditions in the $x$-direction for both the fluid and the structure. The analog of Proposition~\ref{P01} for the clamped plate and Dirichlet boundary conditions for the fluid velocity was proven in \cite[Theorem~1]{G}, in the case of a nonlinear plate in \cite[Theorem~2.1]{TW}, for the dynamic pressure boundary conditions for the fluid and a nonlinear shell in \cite[Theorem~15]{MC2}, and for periodic boundary conditions for the structure, where the fluid is fully enclosed in the structure (so that there are no fluid boundary conditions) and the structure is a nonlinear Koiter shell, in~\cite[Theorem~1.2]{MS}. 

The construction of weak solutions in \cite{TW,MC2,MS} is based on the methodology developed in \cite{MC1}, which decouples the problem into the fluid and structure parts, and is therefore insensitive to the precise form of the fluid boundary conditions, provided that they yield a well-posed fluid problem and an energy inequality for the coupled system. Note that periodic boundary conditions satisfy these requirements. Consequently, the construction of weak solutions in our periodic setting follows along the same lines as in~\cite{TW,MC2,MS}. For details on how to adapt this construction scheme, we refer the reader to \cite[Section~6.1]{MS}, where a more general geometry and more general structure models are treated, with the fluid fully enclosed in the structure (and hence without fluid boundary conditions). Also, note that the proof of the blow-up alternative is based on energy estimates and also do not depend on the fluid boundary conditions.
\end{proof}

Now we are ready to prove the main result.

\begin{proof}[Proof of Theorem~\ref{T01}]
Assume that there does not exist a global weak solution.
By Proposition~\ref{P01}, there exists a weak solution $(u,\eta)$ on $(0,T)$ in the sense of Definition~\ref{D01}, and $T$ is the first contact time.
%
%
We aim to prove that a contact does not occur at~$T$.
Using the incompressibility condition \eqref{EQNS}$_2$ and the kinematic coupling \eqref{EQ26}$_2$, we obtain
\begin{align}
\begin{split}
	&
	\frac{d}{d t}
	\int_{\mathbb{T}^2}
	\eta (\bx,t)\dd \bx
	=
	\int_{\mathbb{T}^2}
	\pt \eta (\bx,t)\dd \bx
	=
	\int_{\mathbb{T}^2} 
	u_3 (\bx, \eta (\bx,t), t)
	\dd \bx
	\\&
	=
	\int_{\partial {\Omega}_{\eta} (t)}
	u(\bx,z, t) \cdot n^\eta \dd \bx \dd z
	=
	\int_{{\Omega}_\eta (t)} \diver u (\bx, z) \dd \bx \dd z
	=0,
	\llabel{EQcons}
\end{split}
\end{align}
which leads to
\begin{align}
	\bar{\eta}_0
	=
	\int_{\mathbb{T}^2} \eta_0 (\bx) \dd \bx
	=
	\int_{\mathbb{T}^2} \eta (\bx,t) \dd \bx
	\comma 
	0\leq t < T.
   \label{EQ01}
\end{align}

From the energy inequality \eqref{EI}, we have
\begin{align}
    E(t) + \int_0^t D(s) \leq E(0) 
    \comma 
    0\leq t < T,
    \label{EQ10}
\end{align}
where
\begin{align*}
    E(t) = \frac{1}{2}
    \left(
    \| u(t)\|^2_{L^2(\Omega_{\eta}(t) )}
	+
	\|\partial_{t}\eta(t)\|^2_{L^2({\mathbb{T}^2})}
	+
	\|\Delta_{\bx} \eta(t)\|^2_{L^2({\mathbb{T}^2})}
	\right),
\end{align*}
represents the energy
and
	\begin{align*}
	D(t) = 
	\Vert \nabla u \Vert_{L^2 (\Omega_{\eta} (t))}^2
	+ 
	\aa
	\|\partial_{t}
	\nabla_{\bx}\eta(t)\|^2_{L^2({\mathbb{T}^2})}
\end{align*}
is the dissipation.
Using the Sobolev embedding $H^2({\mathbb{T}^2})\hookrightarrow
L^\infty({\mathbb{T}^2})$
and the standard $H^2$-regularity, there exists a constant $C_S>0$ such that
\begin{align}
	\|\eta(t) - \bar\eta(t)\|_{L^{\infty}({\mathbb{T}^2})} 
	\leq 
	C_S \Vert  \Delta_{\bx} \eta (t)\Vert_{L^2 (\tt)}
	\comma 
	0\leq t < T,
	\label{EQ100}
\end{align}
where $\bar\eta(t) = \int_{\mathbb{T}^2} \eta(t) \dd x$.
Since, by \eqref{EQ01},
$\bar{\eta} (t)= \int_{\mathbb{T}^2} \eta_0 (x) \dd  x = \bar{\eta}_0$, it follows from \eqref{EQ10} and \eqref{EQ100} that
\begin{equation*}
	\|\eta(t) - \bar{\eta}_0\|_{L^\infty({\mathbb{T}^2})} 
	\leq 
	C_S\|\Delta_{\bx} \eta(t)\|_{L^2({\mathbb{T}^2})}
	\leq
	C_S \sqrt{2E(t)}
	\leq
	C_S \sqrt{2E(0)}
	\comma
	0\leq t < T.
\end{equation*}
Therefore, we have
\begin{equation}
    \eta(x,t) 
    \geq 
    \bar{\eta}_0 
    -
    C_S \sqrt{2E(0)}
    \comma x\in \tt, \quad {0\leq t < T}.
   \label{EQ02}
\end{equation}
Choosing the constant $C_\kappa = C_S^2/(1-\kappa)^2$ and initial data satisfying \eqref{cond_nonc}, i.e., 
\begin{align}
	2E(0) \leq \frac{\bar{\eta}_0^2}{C_\kappa},
	\label{EQ11}
\end{align}
from \eqref{EQ02} and \eqref{EQ11} it follows that
\begin{equation*}
    \eta(x,t) 
    \geq
    \kappa \bar{\eta}_0>0
    \comma x\in \tt,\quad {0\leq t < T}.
\end{equation*}
Thus, we conclude that contact does not occur at the time~$T$.
This contradicts the definition of $T$, by Proposition~\ref{P01}, and
thus the weak solution $(u,\eta)$ is global.
\end{proof}

We conclude this section with a remark.

\begin{remark}
{\rm
In applications, the choice of $\kappa\in(0,1)$ in Theorem~\ref{T01} controls the guaranteed distance to the contact.
}
\end{remark} 

\section{Applications to microfluidics}\label{sec:app}
We start with the geometry of the three-dimensional periodic channel with deformable top boundary. At every time $t\geq 0$, the fluid occupies the domain 
\begin{equation*}
{\Omega}_{\eta}^L (t)=
\{ (\bx, z) : \bx=(x_1,x_2) \in \omegaL,
z\in (0,\eta(\bx,t))\}\subseteq\R^3,
\end{equation*}
of lateral lengths $L>0$,
where $\omegaL=\R^2/(L\Z)^2$ is the two-dimensional flat torus of side length~$L$, and $\eta \colon \omegaL\times [0,+\infty) \to \R$ represents the height of the moving boundary.
The dimensional Navier-Stokes equations describing the flow of a 
viscous fluid in $\Omega^L_{\eta} (t)$ are given by
\begin{align}
\begin{split}
\varrho_f\left(\partial_{t} u  + u\cdot \nabla u \right) - \mu\Delta u  + \nabla p &= 0 \onon{\Omega_{\eta}^L (t) \times (0,T)}, \\
\diver u &= 0 \onon{\Omega_{\eta}^L (t) \times (0,T)},
\end{split}
\label{EQdNS}
\end{align}
where $\varrho_f$ is the fluid density and $\mu$ is the viscosity.
For simplicity we neglect the viscoelasticity and assume that the dynamics of the top boundary is governed by the dimensional plate equation 
\begin{equation}
\varrho_s b\etah_{tt}
+
B\Delta_{\bx}^2 \etah
= G \onon{\omegaL\times (0,T)}
,
\label{EQdPL}
\end{equation}
where 
\begin{equation}
G
=
-S_{\eta}
(-p I+\mu\nabla u) (\bx, \eta (\bx,t),t)
n^{\eta}\cdot e_3
\onon{\omegaL\times (0,T)}.
\label{EQ82b}
\end{equation}
Above, $\varrho_s$ denotes the structure density, $b$ is the thickness of the top plate, and $B$ is the bending coefficient given by
$B = Eb^3/12(1 - \nu^2)$, where $E$ is the Young's modulus and $\nu$ the Poisson ratio
of the plate. 

We non-dimensionalize the system \eqref{EQdNS}--\eqref{EQ82b}
in the standard way. The geometry of the channel is 
non-dimensionalized
by
\begin{align*}
  \hat{x} = \frac{x}{L} ,\quad
  \hat{z} = \frac{z}{L} ,\quad
  \hat t = \frac{t}{T}  ,\quad
  \hat u = \frac{u}{U}  ,\quad
  \hat{p} = \frac{p}{P},\quad
  \hat\eta = \frac{\eta}{L}.
\end{align*}
Setting the time and pressure scales as
\begin{equation*}
T = \frac{L}{U},\quad P = \frac{\mu U}{L}
\end{equation*}
leads to the non-dimensionalized Navier-Stokes equations
\begin{align}
\begin{split}
\Rey\left(\partial_{\hat t} \hat u  + \hat u\cdot \hat\nabla \hat u \right) - \hat\Delta \hat u  + \hat\nabla \hat p &= 0 \onon{\Omega_{\hat\eta}^1 (\hat t) \times (0,\hat T)}, 
\\
\hat{\diver} ~ \hat u &= 0 \onon{\Omega_{\hat\eta}^1 (\hat t) \times (0,\hat T)},
\end{split}
\label{EQndNS}
\end{align}
where $\Rey = \varrho_f UL/\mu$ denotes the Reynolds number. This is precisely our starting system \eqref{EQNS} with $\Rey = 1$. 

Similarly, the plate equation turns into 
\begin{equation}
\frac{\varrho_s b U}{\mu}\hat \eta_{{\hat t}{\hat t}} 
+ \frac{B}{\mu U L^2}\hat \Delta_{\hat x}^2\hat \eta  
= \hat G
 \onon{\tt\times (0,\hat T)}.\label{EQndPL}
\end{equation} 
Introducing the dimensionless numbers 
\begin{align}
\rho = \frac{\varrho_sbU}{\mu}\quad\text{and}\quad \beta = \frac{B}{\mu U L^2},
\end{align}
and then neglecting hats in further derivations, the energy of the system \eqref{EQndNS}--\eqref{EQndPL} equals
\begin{equation*}
E(t) =  \frac{1}{2}
    \left(
    \Rey\| u(t)\|^2_{L^2(\Omega_{\eta}(t) )}
	+
	\rho\|\partial_{t}\eta(t)\|^2_{L^2({\mathbb{T}^2})}
	+
	\beta\|\Delta_{\bx} \eta(t)\|^2_{L^2({\mathbb{T}^2})}
	\right).
\end{equation*}
Following the proof of Theorem \ref{T01}, the condition \eqref{EQ11} becomes
\begin{align}
	\frac{2}{\beta}E(0) \leq \frac{\bar{\eta}_0^2}{C_\kappa},
	\label{EQ111} 
\end{align}
which reads
\begin{align}
		\frac{\Rey}{\beta}\| u_0 \|^2_{L^2({\Omega}_{\eta}(0))}
		+
		\frac{\rho}{\beta}\|\eta_1\|^2_{L^2({\mathbb{T}^2})}
		+
		\|\Delta_{\bx}\eta_0 \|^2_{L^2({\mathbb{T}^2})}
		\leq \frac{\bar{\eta}_0^2}{C_\kappa}.
		\label{cond_nonc_nd}
	\end{align}

We now discuss \eqref{cond_nonc_nd} in the context of microfluidic channels. For this purpose, we introduce 
the relative channel thickness $\eps = H/L$, where $H$ is the height of the channel. In the microfluidic setup, $\eps$ is typically a small parameter. 
Note that above we rescaled the channel height also by $L$; therefore, an assumption that $\eta_0(x) = \eps h_0(x)$ with $\bar h_0 =  \int_{\TT^2} h_0 (x)\dd x=1$ expresses the small initial height of the channel. 

Assume that the fluid is initially in the laminar regime, close to the Poiseuille flow, i.e., $u_0(x,z) \approx z(z-\eta_0)\tilde u_0(x)$, and the first term on the left hand side of \eqref{cond_nonc_nd}  is bounded by means of the corresponding norm of the Poiseuille velocity, i.e.,
\begin{equation}
	\frac{\Rey}{\beta}\| u_0 \|^2_{L^2({\Omega}_{\eta}(0))} \leq \frac{C\Rey\eps^5}{\beta}\| h_0^{5/2}\tilde u_0 \|^2_{L^2(\tt)},
\end{equation}
for some $C>0$, independent of all quantities.
Using the definition of the Reynolds number, the parameter $\rho$ can be written as
\begin{equation*}
\rho = \frac{\varrho_s}{\varrho_f}\frac{b}{L}\Rey,
\end{equation*}
and we may assume that $\rho\leq \Rey \eps$ due to $b\ll L$  
and with $\varrho_s$ comparable to~$\varrho_f$. 
Since $\eta_1(x) = u_0(x,\eta_0(x))\cdot e_3$, assuming that $\pa_z u_0(x,z)\approx (2z - \eta_0)\tilde u_0(x)$, the trace theorem on a thin domain of thickness $O(\eps)$ implies
\begin{align}
\begin{split}
\frac{\rho}{\beta}\|\eta_1\|^2_{L^2({\mathbb{T}^2})} &\leq \frac{C\Rey\eps^5}{\beta}\left(\| h_0^{5/2}\tilde u_0 \|^2_{L^2(\tt)} +\| h_0^{3/2}\tilde u_0 \|^2_{L^2(\tt)}\right. \\
&\left.\qquad\qquad 
+~\| h_0^{5/2}\nabla_x\tilde u_0 \|^2_{L^2(\tt)}+ \| \nabla_xh_0\tilde u_0 \|^2_{L^2(\tt)}\right),
\end{split}
\end{align}
where $C>0$ now depends on~$h_0$.
Finally, the third term on the left-hand side of \eqref{cond_nonc_nd} becomes $\eps^2\|\Delta_{\bx}h_0 \|^2_{L^2({\mathbb{T}^2})}$, while the right-hand side reads $\eps^2/C_\kappa$. Observe that the first two terms in \eqref{cond_nonc_nd} are dominated by 
\begin{equation}\label{EQ410}
\frac{C\Rey\eps^5}{\beta},
\end{equation}
where $C>0$ depends on $h_0$, $\tilde u_0$ and their derivatives. In variety of microfluidic applications \cite{CJFY, HoMa04,TBM} the expression in \eqref{EQ410} is of order much less than $O(\eps)$ (cf.~Table \ref{tab1}) and thus the condition \eqref{cond_nonc_nd} essentially reduces to
\begin{equation}\label{EQ411}
\|\Delta_{\bx}h_0 \|^2_{L^2({\mathbb{T}^2})} 
		\leq \frac{1}{C_\kappa} = \frac{(1-\kappa)^2}{C_S^2},
\end{equation}
where $C_S>0$ is the optimal Sobolev constant defined in \eqref{EQ100} and $\kappa\in(0,1)$ controls the distance to the contact, which here equals~$\kappa\eps$. 
\begin{table}
\caption{Non-dimensional numbers for microfluidic channels in corresponding reference.}\label{tab1} 
\centering
\begin{tabular}{lcccc}
\hline
{Reference} & 
$\eps$ & 
$\Rey$ & 
$\beta$ &
$\Rey\eps^5/\beta$\\
\hline
\cite{CJFY} & 
$6.9\times 10^{-3}$ & 
$\lesssim 1$ & 
$15$--$760$ &
$2\cdot 10^{-14}$--$10^{-12}$ \\
 
\cite{TBM} & 
$5\times 10^{-3}$ & 
$\approx 1$ & 
$\approx 15$ &
$2\cdot 10^{-13}$ \\

\cite{HoMa04} & 
$1.67\times 10^{-1}$ & 
$\approx 0.4$ & 
$\approx 3.2$ &
$1.62\cdot 10^{-5}$\\
\hline
\end{tabular}
\end{table}
Since we are on the unit torus, the constant $C_S$ can be explicitly calculated by means of the Fourier method, providing
\begin{equation*}
C_S = \frac{1}{(2\pi)^2}\left(\sum_{k\in \Z^2\setminus\{0\}}\frac{1}{|k|^4}\right)^{1/2}.
\end{equation*}
Thus, approximately, the condition \eqref{EQ411} becomes 
\begin{equation*}
\|\Delta_{\bx}h_0 \|^2_{L^2({\mathbb{T}^2})} 
		\leq 258.598(1-\kappa)^2.
\end{equation*}
Using this perspective, once again we realize that the condition~\eqref{cond_nonc} in Theorem \ref{T01} is not a smallness assumption, but a rather general condition, which gives the global existence result for large body of physical systems in technological applications. 


\section*{Acknowledgment}
MB and BM were supported by the Croatian Science Foundation under
project IP-2022-10-2962, while
IK was supported in part by the NSF grant DMS-2205493.

\end{document}